\documentclass[11pt,reqno]{amsart}
\usepackage{amsmath,amssymb,amsfonts,amsthm}
\usepackage{fullpage}
\usepackage{graphicx}
\usepackage{fixmath}
\usepackage{comment}
\usepackage{hhline}
\usepackage[latin1]{inputenc}
\usepackage[margin=2.5cm]{geometry}

\usepackage{xcolor}
\usepackage{caption}

\usepackage{url}
\usepackage{enumitem}

\usepackage{setspace} 

\numberwithin{equation}{section}

\newtheorem{theorem}{Theorem}[section]
\newtheorem{proposition}[theorem]{Proposition}
\newtheorem{corollary}[theorem]{Corollary}
\newtheorem{lemma}[theorem]{Lemma}
\theoremstyle{definition}

\begin{document}

\title{Directional expansiveness for the chair tiling}

\author{Hyeeun Jang}
\address{Department of Mathematics, Abilene Christain University, Abilene, TX 79601}
\email{hxj22b@acu.edu}

 \begin{abstract}
We study directional expansiveness of the chair tiling. We show that for every direction $V$ and every $r>0$, there exist distinct chair tilings that agree on $\mathbb{Z}^2\cap V^r$. Hence the chair tiling has no expansive direction.
 \end{abstract}


\maketitle

A chair tile is a tile made by joining three squares in an L-shape. Rotating this tile by $90^\circ$ gives tiles in four different orientations (Figure~\ref{chairtiles}). Figure~\ref{substitution} shows the substitution rule. Each tile is divided into four chair tiles according to the substitution rule. The one in the middle keeps the same orientation as the original tile, and the other three fill the three corners of the L-shape.

\begin{figure}[h]
\centering
\begin{minipage}[b]{0.5\linewidth}
\centering
\includegraphics[width=0.85\linewidth]{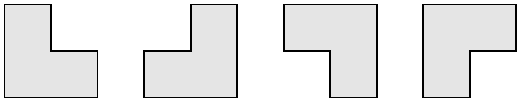}
\caption{The chair tiles}
\label{chairtiles}
\end{minipage}
\hfill
\begin{minipage}[b]{0.47\linewidth}
\centering
\includegraphics[width=0.75\linewidth]{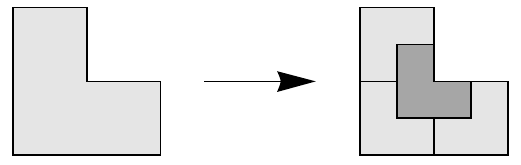}
\caption{The chair substitution.}
\label{substitution}
\end{minipage}
\end{figure}

As we repeat this process, small chairs form larger chair shapes, and those tiles again form larger tiles, so that a repeating pattern is created. If we expand the result back to the original size each time, we obtain larger and larger patches of chairs. Let us imagine covering the plane with these chair tiles. Three squares come together to form one chair, and four chairs come together to form a larger chair, so if we know the arrangement of the small tiles, we can group them into larger chairs, and those into still larger ones (Figure~\ref{hierarchy}). Now let us reverse the question. Suppose we do not know the whole tiling, and we can observe the tiling only through a long strip. Suppose we may choose both the width and the direction of this strip freely, and we can see only the tiles inside that strip. The width may be made as large as we like, but once it is chosen, it cannot be changed. Then, just by looking at the information about the tiles inside that large strip, can we reconstruct the whole chair tiling? This question is a question about directional expansiveness. Boyle and Lind~\cite{BL} defined directional expansiveness for $\mathbb{Z}^d$-actions. We define a direction $V$ to be a line through the origin.  A direction $V$ is called \emph{expansive} if there is a width $r>0$ such that any two tilings that agree inside the strip $V^r$ are equal everywhere.

\begin{figure}[h]
\centering
\begin{minipage}[b]{0.40\textwidth}
\centering
\setlength{\captionindent}{0pt}
\includegraphics[width=\linewidth]{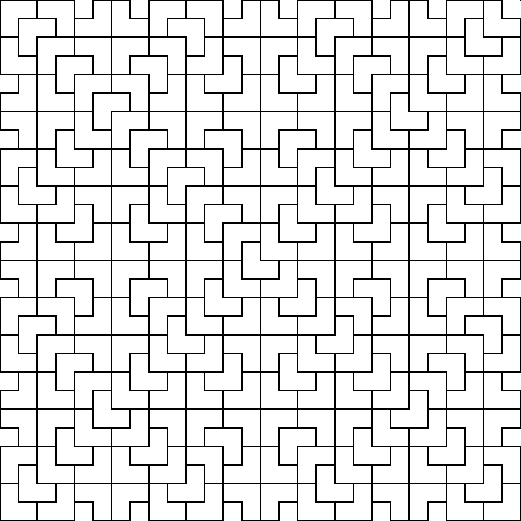}
\caption{A patch of a chair tiling.}
\label{chair}
\end{minipage}\hfill
\begin{minipage}[b]{0.56\textwidth}
\centering
\setlength{\captionindent}{0pt}
\includegraphics[width=\linewidth]{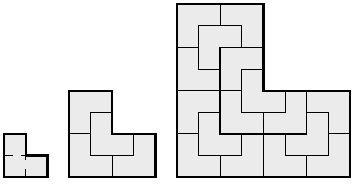}
\caption{Larger chairs formed by repeated substitution.}
\label{hierarchy}
\end{minipage}
\end{figure}

They also introduced this notion and proved two facts. The set of expansive directions is open, and the set of non-expansive directions is always non-empty. Jang and Robinson~\cite{JR} studied directional expansiveness for $\mathbb{R}^d$-actions and for Penrose tilings. They showed that for the Penrose tiling dynamical system, which has finite local complexity (FLC), there are exactly five non-expansive directions, the directions perpendicular to the 5th roots of unity \cite{JR}. From this point of view, we might expect a similar property for the chair tiling. This is because the chair uses a single tile shape and its substitution rule can also be written as a $2\times2$ pattern. But in fact this is not the case. What we will see in this note is that the chair tiling has no expansive direction. Whatever direction we choose, and however wide we make the strip, we can find two chair tilings that look the same inside the strip but are different on the whole plane. 

We follow the notions in Robinson's paper~\cite{Rob}. The chair has four orientations, shown in Figure~\ref{chairtiles}. These four chairs form the set of \emph{prototiles}, which we denote by $\pi$. For each $\tau\in\pi$, let $\sigma(\tau)$ be the tiling of the doubled chair $2\tau$ by four translates of prototiles in $\pi$, as shown in Figure~\ref{substitution}. Robinson~\cite{Rob} calls $(\pi,\sigma)$ a \emph{rep-tile set}, and we call $\sigma$ the \emph{substitution}. We apply the substitution starting with a single chair. Let $x_1=\{\tau\}$ for some $\tau\in\pi$, and define \[ x_{n+1}=\sigma(x_n), \] where $\sigma$ is applied to each tile of $x_n$ (see Figure~\ref{hierarchy}).  Each $x_n$ is a patch, that is, a finite set of tiles whose union is connected. A \emph{chair tiling} is a tiling $x$ of $\mathbb{R}^2$ by translates of tiles in $\pi$ such that every patch in $x$ is, up to translation, a patch in $x_n$ for some $n$. We write $X_c$ for the set of all chair tilings~\cite{Rob}. Figure~\ref{chair} shows a part of a chair tiling. Since each chair is made of three unit squares, we mark each square with an arrow pointing to the point where the three squares of the chair meet (see Figure~\ref{arrows}). We name the four arrows
\[
p=\nearrow,\qquad q=\nwarrow,\qquad r=\swarrow,\qquad s=\searrow,
\]
and let $\mathcal{A}=\{p,q,r,s\}$. We call $p,q,r,s$ \emph{arrows}, or equivalently, \emph{letters}.

\begin{figure}[h]
\centering
\includegraphics[width=7cm, height=4cm]{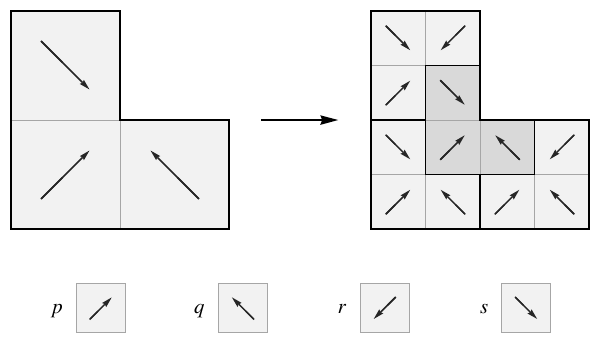}
\caption{The chair substitution with arrows, and the four letters $p, q, r, s$.}
\label{arrows}
\end{figure}

Now we ask whether we can recover the chairs from the arrows. To find a chair, we look for a point where three arrows meet. The three corresponding squares form one chair, and every arrow belongs to exactly one such chair. Robinson~\cite{Rob} calls this the \emph{three arrow rule}. Thus the arrows determine the chair tiling.

 For $x\in X_c$, let $\mathbf{t}\in[0,1)^2$ be such that the unit squares that make up the tiles of $x$ are \[Q_{\mathbf v}=[0,1)^2+\mathbf{t}+\mathbf{v}, \qquad \mathbf{v}\in\mathbb{Z}^2.\] Define $\delta(x):\mathbb{Z}^2\to\mathcal{A}$ by letting $\delta(x)(\mathbf{v})$ be the arrow in $Q_{\mathbf v}$. Thus $\delta:X_c\to\mathcal{A}^{\mathbb{Z}^2}$. We have written chair tilings as arrays, and now we do the same for the substitution. A \emph{block} is a map $\chi:B\to\mathcal{A}$, where $B\subseteq\mathbb{Z}^2$ is finite. Under the substitution, each square of a chair is replaced by a $2\times2$ block of four squares, and the arrows in these four squares are determined by the arrow in the original square. We call this rule a \emph{$2\times2$ substitution}. The corresponding substitution is a map
\[
\sigma_c:\mathcal{A}\to\mathcal{A}^{\{0,1\}^2}.
\] The substitution $\sigma_c$ is given by Equation~(5.2) in \cite{Rob}:
\begin{equation}\label{equation1}
p\mapsto\begin{smallmatrix} s & p\\ p & q\end{smallmatrix},\qquad
q\mapsto\begin{smallmatrix} q & r\\ p & q\end{smallmatrix},\qquad
r\mapsto\begin{smallmatrix} s & r\\ r & q\end{smallmatrix},\qquad
s\mapsto\begin{smallmatrix} s & r\\ p & s\end{smallmatrix}.
\end{equation}

For example, in Figure~\ref{arrows}, the lower left square of the large chair carries the arrow $p$. After the substitution, this square is replaced by the block 
$\begin{smallmatrix} s & p\\
p & q
\end{smallmatrix}$. Since $p=\nearrow$, the two squares on the diagonal from lower left to upper right keep the arrow $\nearrow$. The remaining two arrows point toward the center of the block, giving $\searrow$ in the upper left and $\nwarrow$ in the lower right. The other three rules follow in the same way. Applying $\sigma_c$ repeatedly to a letter $a\in\mathcal{A}$ gives blocks $\sigma_c^n(a)$ of size $2^n\times 2^n$. Let
\[X_c'=\delta(X_c) \] be the set of arrays obtained from chair tilings. By Proposition~6.2 of \cite{JR}, $X_c$ and $X_c'$ have the same expansive directions. From now on, we work with arrays.


We now construct two arrays in $X_c'$. Begin with the block $\sigma_c(p)=\begin{smallmatrix} s & p\\ p & q\end{smallmatrix}$ with its center at the origin. After applying $\sigma_c$, the four letters nearest the origin remain the same since each letter points along the diagonal through its corner, and the rule keeps that letter in the square nearest to the origin, as can be checked from \eqref{equation1}. So we can continue applying $\sigma_c$ repeatedly. At every repetition, the new block surrounds the previous one, so we get larger blocks, one inside the next, whose union is the entire plane. We denote the resulting array by $x_p$. If we begin with $\sigma_c(q)=\begin{smallmatrix} q & r\\ p & q\end{smallmatrix}$, the same construction produces another array, which we call $x_q$. Every block occurring in $x_p$ or $x_q$ is, up to translation, contained in some $\sigma_c^n(a)$ where $a=p$ or $a=q$ respectively. Therefore both arrays belong to $X_c'$ \cite{Rob}.

\begin{lemma}\label{V}
The arrays $x_p$ and $x_q$ differ only along two diagonal lines
extending upward from the origin, one to the right and one to the
left. These lines form a V. In particular, $x_p$ and $x_q$ agree
everywhere in the lower half plane.
\end{lemma}

\begin{proof}
We start with blocks $\sigma_c(p)$ and $\sigma_c(q)$ that have the same two lower corners, $p$ and $q$. Hence $x_p$ and $x_q$ agree in the lower half plane. The upper right corners of $\sigma_c(p)$ and $\sigma_c(q)$ are $p$ and $r$, and the upper left corners are $s$ and $q$. For the upper right, \eqref{equation1} shows that $\sigma_c(p)$ and $\sigma_c(r)$ differ only along the diagonal in the direction $\nearrow$. By induction, the same is true for $\sigma_c^n(p)$ and $\sigma_c^n(r)$ for every $n$, since by \eqref{equation1} the two blocks agree off the diagonal and carry $p$ and $r$ again on the diagonal. Similarly, $\sigma_c(s)$ and $\sigma_c(q)$ differ only along the diagonal in the direction $\nwarrow$, and so do $\sigma_c^n(s)$ and $\sigma_c^n(q)$. Therefore $x_p$ and $x_q$ differ only along the two diagonal lines forming the V.
\end{proof}

\begin{figure}[h]
\centering
\includegraphics{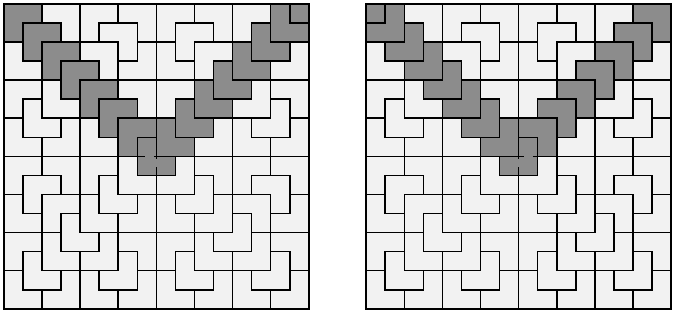}
\caption{Parts of the chair tilings $x_p$ (left) and $x_q$ (right). The shaded chairs are the ones that differ, and they form the two worms of the V.}
\label{worm}
\end{figure}

Figure~\ref{worm} shows the same thing with chairs. The different chairs lie along the two lines of the V, while all other chairs are the same. We call such a line of chairs a \emph{worm}, by analogy with the worms in Penrose tilings in Section 4.5 of \cite{JR}. In the terms of \cite{Rob}, each line of the V is a half-infinite worm, since it starts at the origin and extends in one direction only. Flipping a worm changes one chair tiling into another, as shown in~\cite{Rob}. Up to translation, $x_p$ and $x_q$ form a corner pair in the sense of Example 4.7 of~\cite{BO}.

 The V lies above the origin, so the two tilings agree on the entire lower half plane. Recall that we consider an array $x\in X_c'$ as a function $x:\mathbb{Z}^2\to\mathcal{A}$, and write $x(\mathbf{v})$ for the letter at the position $\mathbf{v}\in\mathbb{Z}^2$. For $x,y\in X_c'$, define \[ D(x,y)=\{\mathbf{v}\in\mathbb{Z}^2: x(\mathbf{v})\ne y(\mathbf{v})\}. \]

\begin{proposition}\label{halfplane}
Let $x\ne y$ be two arrays in $X_c'$. Suppose that $D(x,y)$ is disjoint from an open half plane $H$ whose boundary is parallel to a
direction $V$. Then $V$ is not expansive.
\end{proposition}\begin{proof}
Let $r>0$. Since the boundary of $H$ is parallel to $V$, we can move the strip $V^r$ far enough into $H$. Thus there is some $\mathbf{t}\in\mathbb{Z}^2$ such that $V^r+\mathbf{t}\subset H.$ Since $D(x,y)$ is disjoint from $H$, the arrays $x$ and $y$ agree at every lattice point of $V^r+\mathbf{t}$. Now translate both arrays by$-\mathbf{t}$. The translated arrays still belong to $X_c'$, they are different, and they agree at every point of $\mathbb{Z}^2\cap V^r.$ Since $r>0$ was arbitrary, no width determines the whole array in the direction $V$. Therefore $V$ is not expansive.
\end{proof}

We now use Lemma~\ref{V} and Proposition~\ref{halfplane}
to prove the following~\ref{main}
\begin{corollary}\label{main}
The chair tiling has no expansive direction.
\end{corollary}

\begin{proof}
Let $V$ be any direction. Suppose that $V$ has slope $m$ with $|m|\le 1$. By Lemma~\ref{V}, $x_p$ and $x_q$ differ only on the V. Let \[ H=\{(i,j)\in\mathbb{R}^2:j<mi-1\}. \] Since $|m|\le 1$, the line $j=mi-1$ does not cross the V because the two lines of the V lie on $j = i$ and $j = -1-i$, which stay above $j =-1 +mi$. Then $x_p$ and $x_q$ agree at every point of $H$, and the boundary line of $H$ is parallel to $V$. Therefore, by Proposition~\ref{halfplane}, $V$ is not expansive. Now suppose that $V$ is vertical or has slope $|m|>1$. By the $90^\circ$ symmetry of the chair tiling, the same argument applies to these directions. Therefore no direction is expansive for $X_c'$. By Proposition 6.2 of~\cite{JR}, $X_C$ and $X'_c$ have the same expansive directions. Therefore, the chair tiling has no expansive direction.
\end{proof}

\end{document}